\documentclass[12pt]{article}
\usepackage{amsmath}
\usepackage{amsfonts}
\usepackage{amssymb}
\usepackage{amsthm}
\usepackage{xspace}
\usepackage{geometry}
\theoremstyle{definition}
\newtheorem{lem}{Lemma}

\newtheorem{prop}{Proposition}
\newtheorem{thm}{Theorem}
\newtheorem*{remark}{Remark}

\newcommand{\ep}{\varepsilon}

\newcommand{\be}{\beta}

\newcommand{\del}{\delta}

\newcommand{\lam}{\lambda}

\newcommand{\lt}{L}
\newcommand{\bp}{b}
\newcommand{\abs}[1]{\left\vert#1\right\vert}
\newcommand{\Lone}[1]{\left\vert#1\right\vert_{L^1}}
\newcommand{\infnorm}[1]{\left\vert#1\right\vert_{L^\infty}}
\newcommand{\norm}[1]{\left\Vert#1\right\Vert}
\newcommand{\var}[1]{\text{var}(#1)}

\newcommand{\bv}{\text{BV}}
\newcommand{\ts}{t}
\newcommand{\T}{\mathcal{T}}

\newcommand{\brA}{{\bar A}}
\newcommand{\brn}{{\bar n}}

\newcommand{\cB}{\mathcal{B}}
\newcommand{\cP}{\mathcal{P}}
\newcommand{\cQ}{\mathcal{Q}}

\newcommand\Prob{{\mathbb{P}}}
\newcommand{\bA}{\mathbf{A}}
\newcommand{\bD}{\mathbf{D}}
\newcommand{\pr}{\mathbb{P}^r}
\renewcommand{\L}{\mathcal{L}}
\newcommand{\C}{\mathcal{C}}

\newcommand{\mc}[1]{\mathcal{#1}}
\newcommand{\acim}{ACIM\xspace}
\newcommand{\acims}{ACIMs\xspace}
\newcommand{\leb}{\mathrm{Leb}}

\usepackage[pdftex]{graphicx,color}  

\title
{The Diffusion Coefficient for Piecewise Expanding Maps of the Interval with Metastable States}

\begin{document}


\title{THE DIFFUSION COEFFICIENT FOR PIECEWISE EXPANDING MAPS OF THE INTERVAL WITH METASTABLE
STATES}

\author{Dmitry Dolgopyat\thanks{Department of Mathematics, University of Maryland, College Park, MD 20742.  Email:  dmitry@math.umd.edu.  DD is partially supported by the NSF. } \quad and \quad Paul Wright\thanks{Department of Mathematics, University of Maryland, College Park, MD 20742.  Email:  paulrite@math.umd.edu.  PW is partially supported by an NSF Mathematical Sciences Postdoctoral Research Fellowship.}}

\maketitle

\begin{abstract}
Consider a piecewise smooth expanding map of the interval possessing
several invariant subintervals and the same number of
ergodic absolutely continuous invariant probability measures (ACIMs).
After this system is perturbed to make the subintervals lose their invariance in such a way that
there is a unique ACIM, we show how to approximate the diffusion coefficient for
an observable of bounded variation by the diffusion coefficient of a related continuous time Markov chain.
\end{abstract}

\emph{Dedicated to Manfred Denker on the occasion  of his 60th birthday.}\footnote{We thank Carlangelo Liverani for useful comments on a preliminary version
of this paper.}

\maketitle

\section{Introduction}\label{S:intro}
Metastable dynamics arise in a number of physical systems. In such systems, the phase space can be divided into a finite number of components, called metastable states, that are nearly invariant under the dynamics.  A typical trajectory will remain in one metastable state for an extended period of time before escaping to another metastable state and repeating this behavior.
Rigourous results about metastability in dynamical systems perturbed by noise can
be found in \cite{FreidlinWentzell} and \cite {KellerLiverani09}.

In this paper we are concerned with reducing the description of purely deterministic chaotic systems to cooresponding finite state Markov chains. In particular, we continue the study of the dynamics of hyperbolic interval maps with metastable states initiated in \cite{GonzalezTokmanHuntWright}.
These systems arise from perturbing an initial system
$T_0$ with $m$ disjoint invariant intervals $I_1, I_2, \dots, I_m.$
The initial map has $m$ mutually singular ergodic absolutely continuous invariant measures (\acims),
$\mu_1, \mu_2, \dots, \mu_m.$
$T_0$ is perturbed in such a way that the $\mu_j$
lose their invariance, and the perturbed map $T_\ep$ has only one \acim, $\mu_{\ep}$.
See Figure~\ref{figure:2xmod1}.


\begin{figure}[htbp]
\begin{center}
\label{figure:2xmod1}
\resizebox{7cm}{!}{\includegraphics{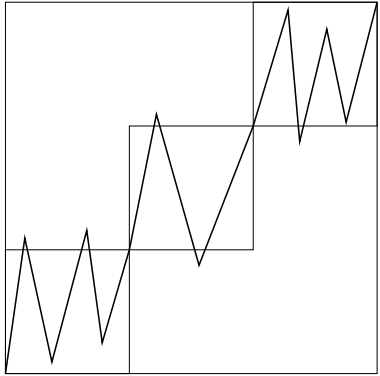}}
    \caption{A map with three almost invariant intervals}
\end{center}
\end{figure}

Such metastable systems can be understood in the context of deterministic dynamical systems with holes
(see \cite{DemersYoung}) as follows.
As the invariance of initially invariant intervals is destroyed
by the perturbation, we think of the small set of points
$I_i\cap T_\ep^{-1}I_j$ that switch from $I_j$ to $I_i$ after the application of $T_\ep$,
as being holes in the initially invariant sets. Therefore the techniques developed to study systems with holes
are useful in our analysis.

As was shown in \cite{GonzalezTokmanHuntWright}, we are able to approximate $\mu_\ep$,
for small $\ep$, by a convex combination $\sum_{j=1}^m p_j \mu_j$
of the initially invariant measures,
where $(p_1, p_2, \dots, p_m)$
is the invariant measure
for the continuous time Markov chain on $m$ states with transition rates
proportional to the asymptotic sizes of the holes. Intuitively, since our map is chaotic on each interval,
we expect the transition times between the $I_j$'s to be almost independent, which explains
the appearence of the above mentioned Markov chain.
In this paper, we show that the Markov approximation also extends to the diffusion matrix for smooth observables.
We believe the methods developed in our paper can be used to describe the transport coefficients in other
chaotic systems, for example, billiards with narrow tunnels
\cite{MachtaZwanzig}. However, in order to present the ideas of our proof in the simplest possible
setting, we restrict our attention here to the setup of \cite{GonzalezTokmanHuntWright}.




\section{Statement of the main results}\label{S:results}

In this section, we define a family of dynamical systems with $m$
nearly invariant (metastable) subsets. They are
perturbations of a one-dimensional piecewise smooth expanding map
with $m$ invariant subintervals $I_1, I_2, \dots, I_m$
of positive Lebesgue measure.
On each of these intervals, the unperturbed system has a unique \acim.
The perturbations break this invariance in such a way that each perturbed
system will have only one \acim.  Our main result is an asymptotic formula for the diffusion coefficent of a smooth observable as the size of the perturbation tends to zero.  We also show that the sequence of jump times in between
different intervals asymptotically approach exponential random variables.

Let $I=[0,1].$ In this paper, a map $T:I\circlearrowleft$ is called
a piecewise $C^2 $ map with $\C=\{ 0=c_0 < c_1<  \cdots < c_d=1\}$
as a critical set if for each $i$, $T\vert_{(c_i,c_{i+1})}$ extends
to a $ C^2$ function on a neighborhood of $[c_i,c_{i+1}]$.  We call $T$ uniformly expanding if its minimum expansion, $\inf_{x\in I\setminus \C_0}|T_0'(x)|$, is greater than 1.  As is customary for piecewise smooth maps, we consider $T$ to be bi-valued at points $c_i\in\C$ where it is discontinuous. In such cases we let
$T(c_i)$ be both values obtained as $x$ approaches $c_i$ from either side.

\subsection{The initial system and its perturbations}\label{S:results_assumptions}

The unperturbed system is a piecewise $C^2$ uniformly expanding map
$T_0:I\circlearrowleft$ with $\C_0=\{0= c_{0,0}< c_{1,0}<  \cdots <
c_{d,0}=1\}$ as a critical set.  There are boundary points $\cB=\{b_j\}\subset(0,1)$
such that
$I_j=[b_{j-1}, b_j]$ ($b_0=0, b_{m}=1$)
i.e. $T_0 (I_j)\subset I_j$.
The existence of an \acim~of bounded variation for $T_0\vert_{I_j}$
is guaranteed by \cite{LasotaYorke}.  We assume in addition:

\noindent \textbf{(I)} \emph{Unique mixing \acims on the initially invariant subsets:}
$T_0\vert_{I_j}$, $j \in\{1\dots m \}$, has only one \acim~$\mu_j$,
whose density is denoted by
$\phi_{j} = d\mu_{j}/dx$.  $(T_0 ,\mu_j) $ is mixing.

From \textbf{(I)}, it follows that all \acims~of $T_0$ are
convex combinations of the ergodic ones, $\{\mu_j\}.$

We define the points in $H_0 = (T_0^{-1} \cB)\setminus \cB$
to be \emph{infinitesimal holes}. (The exclusion of boundary points from the set of infinitesimal holes is not essential, although it does simplify our presentation.  See assumption \textbf{(V)} and the discussion thereafter.)

\noindent \textbf{(II)} \emph{No return of the critical set to the infinitesimal holes:}
For every $k>0$, $(T_0^k \C_0)\cap H_0=\emptyset$.

Since $\phi_j$ are of bounded variation, they can be
suitably defined so that they are  continuous except on at most a countable set of points where
it has jump discontinuities.
Moreover (see \cite[Section 4.2]{GonzalezTokmanHuntWright}), \textbf{(II)}
implies that after $\phi_{j}$ have been so defined, they are
continuous at each of the infinitesimal holes.

\noindent \textbf{(III)} \emph{Positive \acims~at infinitesimal holes:}
$\phi_{j}$  is positive at each of the points in $H_0\cap I_j $.

\textbf{(II)} and \textbf{(III)} are generic conditions, although \textbf{(II)} may be difficult to verify for specific examples.  Similar assumptions are made elsewhere, including in \cite[Section 3.2]{KellerLiverani09} and \cite{GonzalezTokmanHuntWright}.

\noindent \textbf{(IV)} \emph{Restriction on periodic critical points:}
Either
\begin{enumerate}
\item [(a)] $\inf_{x\in I\setminus
\C_0}|T_0'(x)|>2$, or
\item [(b)] $T_0 $ has no periodic critical points, except possibly that $0 $ or $1 $ may be fixed points.
\end{enumerate}

Because $T_0$ may be bi-valued at points in $\C_0$, a critical point $c_{i,0} $ is considered periodic if there exists $n>0 $ such that $c_{i,0}\in T_0^n \{ c_{i,0} \}$.
Condition \textbf{(IV)} is necessary in order to ensure that the transfer operators cooresponding to the perturbed systems defined below satisfy uniform Lasota-Yorke inequalities.  These uniform inequalities are essential for establishing the perturbative spectral results, Propositions \ref{P:secondEvalue} and \ref{P:escaperate}, which are a key ingredient of our proof.
Since we cannot exclude the possibility of the forward orbit of a critical point containing other critical points, these uniform inequalities do not follow directly from the original paper \cite{LasotaYorke}, but rather from later works \cite{BaladiYoung,BaladiBook}, see \cite[Section 4.2]{GonzalezTokmanHuntWright}.

In what follows, we consider $C^2$-small perturbations
$T_{\ep}:I\circlearrowleft$ of $T_0$ for $\ep>0$.  A critical set for $T_{\ep}$ may be chosen as
 $\C_\ep=\{0=c_{0,\ep}< c_{1,\ep}<  \cdots <
c_{d,\ep}=1\}$, where for each $i $, $\ep\mapsto c_{i,\ep}$ is a $C^2 $ function for $\ep\geq 0 $. Furthermore, there exists  $\del>0 $ such that for all sufficiently small $\ep$, there exists a $C^2 $ extension $\hat T_{i,\ep}:[c_{i,0}-\del,c_{i+1,0}+\del]\rightarrow \mathbb{R}$ of $T_{\ep}\vert_{[c_{i,\ep},c_{i+1,\ep}]}$, and $\hat T_{i,\ep} \rightarrow \hat T_{i,0}$ in the $C^2$ topology.

We also assume:


\noindent \textbf{(V)} \emph{The boundary points do not move, and no holes are created near them:}
Precisely, for each $b\in\cB$ we have the following
\begin{itemize}
  \item[(a)] If $b\notin \C_0$, then necessarily $T_0(\bp)=\bp$. We assume further that for all $\ep>0 $, $T_\ep(\bp)=\bp$.
  \item[(b)] If $b\in \C_0$, we assume that $T_{0}(\bp_-)< \bp<T_{0}(\bp_+)$, and also that $b\in \C_\ep$ for all $\ep $.
\end{itemize}
This boundary condition can be considerably relaxed by suitably redefining
the holes discussed below, as was explained in \cite[Section 2.4]{GonzalezTokmanHuntWright}.
For simplicity of presentation, we do not detail these generalizations here.

Set $H_{ij,\ep}=I_i\cap
T_{\ep}^{-1}(I_j).$
We refer to these sets as \textit{holes}.
Once a $T_\ep$-orbit enters a hole, it leaves one of the invariant
sets for $T_0$ and continues in another. As $\ep\rightarrow 0$,
the holes converge  (in the Hausdorff metric) to the infinitesimal holes from which they arise.
Our assumptions imply that there exist numbers $\beta_{ij}\geq 0 $ such that
\begin{equation}\label{E:holesize}
    \mu_i(H_{ij,\ep}) = \ep \be_{ij} +o (\ep) .
\end{equation}
Conisider a continuous time Markov chain with states $1, 2, \dots, m$ and jump rates
from state $i$ to state $j$ equal to $\beta_{ij}.$

\noindent \textbf{(VI)} \emph{Irreducibility:}
The Markov chain defined above is irreducible.

Condition \textbf{(VI)} implies that for
small $\ep>0$, $T_{\ep}$ has only one \acim~$\mu_\ep$, with density
$\phi_\ep = d\mu_\ep/dx$.

Examples of families $T_\ep$ that satisfy \textbf{(I)} through \textbf{(VI)} can be found in \cite[Section 2.4]{GonzalezTokmanHuntWright}.

It is natural to inquire if the family of \acims $\mu_\ep$ has a unique limit as $\ep \rightarrow 0$, and if this limit exists, how to express it as a convex combination of $\mu_1, \dots, \mu_m$.
This problem was successfuly addressed in \cite{GonzalezTokmanHuntWright}.

\begin{prop}[Theorem~1 in~\cite{GonzalezTokmanHuntWright}]\label{P:limitingdensity}
As $\ep\rightarrow 0$,
\[
    \phi_\ep\overset{L^1}{\longrightarrow} \phi_0 =\sum_j p_j \phi_j
\]
where $(p_1, p_2, \dots, p_m)$ is the invariant measure for the Markov chain.
\end{prop}
\noindent We define $\mu_0 $ to be the measure whose density is $\phi_0$.

\subsection{The main results}\label{S:results_theorem}

Proposition \ref{P:limitingdensity} shows that the limiting Markov chain provides useful information
about our metastable system. In this section we prove two additional results making the connection between the
dynamical system and the Markov chain more precise.

To state our first result we need to define, both for the $T_\ep $-dynamics and for the Markov chain,
a sequence of times that indicate when a transition occurs between  different $I_j$'s.

\noindent \textbf{For the Markov chain:}
Set $\ts^M_0 =0 $, and for $i>0 $, let $\ts^M_i $ be the $i^{\mathrm{th}} $ time that the Markov dynamics have
changed states.
Then for $i\geq 1 $, set $\T_i^M =\ts^M_i -\ts^M_{i -1} $. Let $z_i^M$ be the state of the chain after the
$i^{\mathrm{th}}$ transition.
Let $\pr$ denote the probability measure constructed on the space $\{1, \dots, m\} ^{[0,\infty)} $
for the Markov chain by starting at $I_r$ at $t =0 $ and then evolving forward by the Markov dynamics.
Then, given the present state of the chain, $\T_i^M $ are exponential random variables.
That is, for $t\geq 0$,
the densities are given by
\[
    d\pr(\T_i^M = t|z_{i-1}^M=j)=
    \be_j e^{-\be_j t} dt
\]
where $\be_j=\sum_k \be_{jk}.$ Also given that $z_{i-1}=j$, $z_i$ is independent of $\T_i^M$
and
\[
\pr(z_i^M = k|z_{i-1}^M=j)=\be_{jk}/\be_j.
\]

\noindent \textbf{For $T_\ep$:} Let $z(x)=k$ if $x\in I_k.$
Set $\ts^\ep_0 =0 $, and for $i>0 $, let $\ts^\ep_i =\inf \{n>\ts^\ep_{i-1}:
z(T_\ep^n x)\neq z(T_\ep^{\ts^\ep_{i-1}} x)  \}$.
Then for $i\geq 1 $, set $\T_i^\ep =\ts^\ep_i -\ts^\ep_{i -1} $.

Our first main result
states that the finite dimensional distributions of jumps of the deterministic systems
converge to the finite dimensional distributions of jumps of the Markov chain.

\begin{thm}
\label{ThDetMC}
Fix $j, p$ and $S$.  For any intervals $\Delta_k=[a_k, b_k],$ and numbers $r_k\in \{1, \dots, m\},$
$k=1,\dots, p$
$$ \mu_j(\ep \T_k^\ep\in \Delta_k, z(\ts^\ep_k)=r_k,
\text{ for }k=1, \dots, p)\to \pr(\T_k\in \Delta_k, z_k^M=r_k, \text{ for }k=1, \dots, p) $$
and the convergence is uniform for $\max b_k \leq S$.
\end{thm}

Consider an observable $A: I \to \mathbb{R}$.  For each fixed small $\ep > 0$, the Ergodic Theorem provides a law of large numbers, i.e. $N^{-1} \sum_{k=0}^{N} A \circ T_\ep^k \rightarrow \mu_\ep(A)$  a.e. and in $L^1$ as $N \rightarrow \infty$.  If $A$ is of bounded variation, the Central Limit Theorem applies  \cite{Liverani95} and states that $N^{-1/2} \sum_{k=0}^N A \circ T_\ep^k$ approaches a normal distribution as $N \rightarrow \infty$.  We let $D_\ep(A)$ be the diffusion coefficent, which is the variance of the limiting normal distribution.

\begin{thm}
\label{thmmain}
For any observable $A$ of bounded variation, \[
 \ep D_\ep(A)\to \bD^M(\bA)
\]
as $\ep\to 0$, where
$\bA$ is the observable on the state space of our Markov chain such that
$\bA(j)=\int_{I_j} A \phi_j dx$ and
$\bD^M$ stands for the diffusion coefficient.
\end{thm}

\begin{remark}
$\bD^M$ can be computed efficiently. Assume for simplicity that $\bA$ has zero mean (otherwise we subtract
a constant from $\bA$). Let  $G$ denote the generator matrix of our Markov process, that is $G_{jk}=\be_{jk}$
if $j\neq k$ and $G_{jj}=-\be_j.$ Then
$$ \bD^M=\sum_{jk} p_j \bA(j) \int_0^\infty p_{jk}(t) \bA(k) dt=
\langle p \bA, \int_0^\infty e^{tG} \bA dt \rangle=
\langle p \bA, G^{-1} \bA\rangle . $$
\end{remark}

The proof of Theorems \ref{ThDetMC} and \ref{thmmain} depend on two perturbation results,
Propositions \ref{P:secondEvalue} and \ref{P:escaperate}. Proposition \ref{P:escaperate} is a variation on a result
of Keller and Liverani \cite{KellerLiverani99, KellerLiverani09}, see Appendix~\ref{S:appendix} for details, while
Proposition \ref{P:secondEvalue} appears to be new. Our paper is organized as follows. In Section~\ref{S:preparation}
we recall the transfer operator approach to the study of piecewise expanding interval maps and state
Propositions \ref{P:secondEvalue} and \ref{P:escaperate}. In Section \ref{S:jumpprocess}
we derive Theorem \ref{ThDetMC} from
Proposition \ref{P:escaperate}. Section \ref{S:ProofMain} contains the derivation of Theorem \ref{thmmain} from
Theorem \ref{ThDetMC} and Proposition \ref{P:secondEvalue}. Finally in Section~\ref{S:ProofSecond} we explain how
Proposition \ref{P:secondEvalue} follows from Theorem \ref{ThDetMC}.

\section{Preparatory material}\label{S:preparation}

\subsection{Function spaces and norms}\label{S:preparation_norms}

We use $\leb $ to denote normalized Lebesgue measure on $I$ and $L^1 $ to denote the space of
Lebesgue integrable functions on $I$.  For $f:I\rightarrow \mathbb{R}$, let $\Lone{f} =\int_I \abs{f(x)}\,dx $, $\infnorm{f}=\sup_{x\in I} \abs{f(x)}$, and $\var{f} $ be the total variation of $f$ over $I$; that is,
\[
\var{f}= \sup\{ \sum_{i=1}^n\abs{f(x_{i})-f(x_{i-1})}: n\geq 1, 0\leq x_0<x_1<\dots<x_n\leq 1 \}.
\]
$\bv=\bv(I)$  is the Banach space of functions $f : I \rightarrow \mathbb{R}$ with norm
\[
    \norm{f}=\inf_{g=f \text{ except on a countable set}}\var{g}+\infnorm{g}.
\]
An element of $\bv$ is technically an equivalence class of functions, any two of which agree except on a countable set; we generally ignore this distinction.

\subsection{Transfer operators and the spectral setting}\label{S:preparation_operators}

For $\ep \ge 0$, let $\L_\ep $ be the transfer operator associated with $T_\ep $ acting on $\bv $, i.e.
\[
    \L_\ep A(x) = \sum_{y\in T_\ep^{-1}\{x\} } \frac{A(y)}{\abs{T_\ep'(y)}}.
\]
Note that $\L_0$ has 1 as an isolated eigenvalue of multiplicity $m$, and then a spectral gap.
Let $P$ denote the associate spectral projection, that is
$$ (PA)(x)=\left(\int_{I_j} A(y) dy \right) \phi_j(x) \text{ if } x\in I_j. $$
For small $\ep>0$, 1 becomes an isolated eigenvalue creating a small spectral gap.
Note that $\L_\ep$ preserves the space $\bv_0$ of $\bv$-functions with zero mean.


\begin{prop}\label{P:secondEvalue}
There exists $\eta<1$ and $\kappa>0$ such that
$$ ||\L_\ep^{\kappa n/\ep}||_{\bv_0} \leq \eta^n. $$
\end{prop}
The proof of Proposition \ref{P:secondEvalue} is given is Section \ref{S:ProofSecond}.

\begin{remark}
In case of two intervals much more precise asymptotics of the spectral gap can be deduced from the results of
\cite{KellerLiverani09}. It is likely that a similar statement holds for an arbitrary number of intervals, however we do
not pursue this question here since the weaker version stated above is sufficient for our purposes. In fact we hope
that the argument used to prove Proposition \ref{P:secondEvalue} can be extended also to the case of piecewise
hyperbolic systems such as billiards consisting of several regions connected by small holes.
\end{remark}

For $j\in\{1, \dots, m\}$, let $\L_{j,\ep} $ be the transfer operator acting on $\bv(I_j) $
 where we consider $H_{j,\ep}$ as a hole, i.e.
\[
    \L_{j,\ep} A(x) = \sum_{y\in (I_j\cap T_\ep^{-1}\{x\})} \frac{A(y)}{\abs{T_\ep'(y)}}.
\]
From \cite{KellerLiverani99} we see that, for small $\ep>0 $, $\L_{j,\ep}$ has an isolated simple eigenvalue  $\lam_{j,\ep}<1$ of multiplicity one with $\lam_{j,\ep}\to 1$ as $\ep\to 0 $, and otherwise the spectrum has a uniform spectral gap.  More precisely, if $\cQ_{j, \ep}$ is the spectral projection corresponding to $\lam_{j,\ep}$, the following
statement holds.

\begin{prop}\label{P:escaperate}
There exists a measure $\nu_{j,\ep}$ and a function $\phi_{j,\ep} \in \bv(I_j)$ such that
$$\nu_{j,\ep}(\phi_{j,\ep})=1\quad, \cQ_{j,\ep} = \nu_{j, \ep} (\cdot) \phi_{j, \ep} $$
and
\begin{itemize}
 \item [(a)] $\lam_{j,\ep}=1-\be_j \ep+o (\ep)$.

 \item [(b)] As $\ep\to 0,$
$\nu_{j,\ep}(A)\to \int A(y) dy$ where the convergence holds in strong topology in $\bv^*$ and
$\phi_{j,\ep}\to \phi_j$ in $L^1.$ Moreover
\begin{equation}
\label{DensityHoles}
\lim_{\ep\to 0} \sup_{H_{j,\ep}} |\phi_{j,\ep}-\phi_j|= 0.
\end{equation}

\item [(c)] There exists $C > 0$ such that, for all small $\ep$, $A \in \bv(I_j)$,
and all $n \ge 0$, $$ \left\Vert \lam_{j,\ep}^{-n} \L_{j,\ep}^n A-\cQ_{j, \ep} A \right\Vert\leq C \theta^n \norm{A}.$$
\end{itemize}
\end{prop}

This proposition follows from the work of Keller and Liverani \cite{KellerLiverani99, KellerLiverani09},
see Appendix \ref{S:appendix} for details.

\section{Convergence of the jump process}\label{S:jumpprocess}

\subsection{Tightness}\label{S:jumpprocess_jumps}
We need to know
that the distribution of jumps is tight so we begin with the following result.

\begin{lem}
\label{Lm2jumps}
Given $S, \delta, p$ there exists $\sigma$ such that, for all $\ep$ sufficently small,
$$\mu_p(\exists k \text{ with } t^\ep_k \le S/\ep \text{ and } \T_{k+1}^\ep \le \sigma/\ep)\leq \delta.$$
\end{lem}

The following estimate plays a key role in our analysis. Given a segment $J \subset I$ let
$r_n(x)$ be the distance of $T_\ep^n x$ to the boundary of the component of $T_\ep^n J$ containing it.  Recall that $\leb$ denotes Lebesgue measure on $I = [0,1]$.
\begin{lem}[Growth Lemma]
(see \cite[Section 5.10]{ChernovMarkarian})
\label{LmGr}
There exists $\Lambda > 1$, $c > 0$ such that, for all $\ep$ small enough, and all $J$, and all $n \ge 0$,
$$ \leb(x: r_n(x)\leq\ep)\leq \leb(x: r_0(x)\leq \ep/\Lambda^n)+c \, \leb(J)\ep. $$
\end{lem}
\begin{proof}[Proof of Lemma \ref{Lm2jumps}]
First, $t_0^\ep = 0$, and if $N = N(\sigma, \ep) = \lceil \sigma / \ep \rceil$, then
$$\mu_p (\mathcal{T}_1^\ep \le \kappa / \ep) \le 1 - \mu_p (\mathcal{T}_1^\ep > N) = 1 - \int_L \L_{\lt,\ep}^N(\phi_\lt) \, dx = 1 - \lambda_{p, \ep}^N \int Q_{\lt,\ep}(\phi_\lt) \, dx + O(\theta^N),$$
where we have used Proposition \ref{P:escaperate} (c).  But from parts (a) and (b) of the same proposition, we see that this can be made arbitrarily small for all small $\ep$ by taking $\sigma$ small enough.

Because $\mu_l\ll \leb$ it suffices to prove the statement for $\leb.$

We follow \cite{Dolgopyat}, Section 18. Let
$S_{n, m}=\sum_{j=n+1}^{n+m} 1_{T_\ep^j x\in H_\ep}.$
We have to show that
$$\sum_{n\leq S/\ep} \int 1_{T_\ep^n x\in H_\ep} 1_{S_{n, \sigma/\ep}(x)>0} dx=o(1), \quad
\ep\to 0, \sigma\to 0.$$
Take a small $r.$ We say that a visit of $x$ to the hole at time $n$
is ($r$-)inessential if the length of the smoothness component
of $T^n_\ep \lt\cap H_\ep$
containing $T^n_\ep x$ is less than $r\ep.$ By Lemma \ref{LmGr} the probability that $x$ will have an inessential
visit to the hole before time $S/\ep$ is less than $CrS$ which can be made as small as we wish by taking $r$ small.
Therefore it suffices to show that for any fixed $r$
$$ \sum_{n\leq S/\ep} \int 1_{\mc{E}_{n,\ep}(x)} 1_{S_{n, \sigma/\ep}(x)>0} dx=o(1) $$
where $\mc{E}_{n, \ep}(x)=\{x$ has essential visit to the hole at time $n\}.$
$$\int 1_{\mc{E}_{n,\ep}(x)} 1_{S_{n, \sigma/\ep}(x)>0} dx=
\leb(\mc{E}_{n,\ep}(x)) \Prob(S_{n, \sigma/\ep}(x)>0|\mc{E}_{n,\ep}(x)). $$
Since $\leb(\mc{E}_{n,\ep}(x))\leq \leb(x: T_\ep^n x\in H_\ep)\leq C\ep$ it suffices to check that
\begin{equation}
\label{CondProb}
\max_{n\leq S/\ep} \leb(S_{n, \sigma/\ep}>0|\mc{E}_{n,\ep}(x))=o(1), \quad \ep\to 0, \sigma\to 0 .
\end{equation}
In order to prove this we observe that due to assumption (II) for any fixed $M_0,$ $S_{n, M_0}(x)=0$
for all $x\in \mc{E}_{n,\ep}$
provided that $\ep$ is small enough. On the other hand if $k\geq M_0$ then by Lemma \ref{LmGr}
applied to $H_{\ep}$
$$ \leb(1_{H_\ep} T_\ep^{n+k} x)>0|(\mc{E}_{n,\ep}(x))\leq C\left(\ep+\Lambda^{-k}\right). $$
Summing over $k\in [M_0+1, \kappa/\ep]$ we obtain \eqref{CondProb}.
\end{proof}

\subsection{Convergence of the finite dimensional distributions}\label{S:jumpprocess_convergence}

\begin{proof}[Proof of Theorem \ref{ThDetMC}]
The proof is by induction on $p$. First, for $p=1$ assume that the initial distribution of $x$ is chosen according to
some density $\rho\in BV(\mu_j).$ Then by Proposition \ref{P:escaperate}(c)
$$ \mu_j(\T_1^\ep=n, z(T^n_\ep x)=r)=\int_{H_{jr}} \L_{j, \ep}^n (\rho) dx=
\lam_{j,\ep}^n \int_{H_{jr}} (\cQ_{j, \ep} \rho) dx+O(\theta^n).$$
Recall that
$$ \cQ_{j, \ep} \rho=\nu_{j, \ep}(\rho) \phi_{j, \ep} .$$
By Proposition \ref{P:escaperate}
$\nu_{j, \ep}(\rho)\to \int \rho(x) dx,$
$\int_{H_{jr}} \phi_{j, \ep}/\ep\to \be_{jr}$
and for $\ep n\approx t,$
$\lam_{j, \ep}^n\to e^{-\be_j t}.$
Thus summation over $n\in \Delta/\ep$ concludes the proof for $p=1.$

Next suppose that the statement is known for some $p.$ Denote
$$ \Omega=\{\T_k^\ep\in \Delta_k, z(\ts^\ep_k)=r_k,
\text{ for }k=1, \dots, p\}. $$
To carry the induction step it is enough to prove that
\begin{equation}
\label{Indpp+1}
\mu_j(\ep \T_{p+1}^\ep \in \Delta_{p+1} z(\ts^\ep_{p+1})=r_{p+1} , \Omega)\to
\mu_j(\Omega) \be_{z_p z_{p+1}} \int_{\Delta_{p+1}} e^{-\beta_{z_p} t} dt.
\end{equation}
Set
$$ (\L_\Omega A)(x)=\sum_{T_\ep^{\ts_\ep^p} y=x} \frac{A(y)}{(T_\ep^{\ts_\ep^p})'(y)} $$
where the sum is taken over $y\in \Omega.$ Then
\begin{equation}
\label{LOmega}
\mu_j((\ep \T_{p+1}^\ep=n z(\ts^\ep_{p+1})=r_{p+1} , \Omega)=
\int_{H_{z_p z_{p+1}}} \L_{j,\ep}^n (\L_\Omega (1_{I_j} \phi_j)) dx.
\end{equation}
Take a small $\sigma$ and rewrite
$$\L_{j,\ep}^n (\L_\Omega (1_{I_j} \phi_j))=\L_{j,\ep}^{n-\sigma/\ep}
\left[\L_{j,\ep}^{\sigma/\ep} (\L_\Omega (1_{I_j} \phi_j))\right] .$$
Due to Lasota-Yorke inequality
$$ \left[\L_{j,\ep}^{\sigma/\ep} (\L_\Omega (1_{I_j} \phi_j))\right] $$
has bounded \bv-norm. Hence arguing as in the $p=1$ case we see that \eqref{LOmega}
is asymptotic to
$$ \be_{z_p z_{p+1}} \int_{\Delta_{p+1}} e^{-\beta_{z_p} t} dt
\int_I \left[\L_{j,\ep}^{\sigma/\ep} (\L_\Omega (1_{I_j} \phi_j))\right] dx (1+o_{\sigma\to 0} (1)). $$
Due to Lemma \ref{Lm2jumps} the last integral here equals to
$$\mu_j(\Omega, \T_{p+1}^\ep>\sigma/\ep)=\mu_j(\Omega)+o_{\sigma\to 0} (1) . $$
Since $\sigma$ is arbitrary this proves \eqref{Indpp+1} completing the proof of Theorem \ref{ThDetMC}.
\end{proof}

\section{Proof of Theorem \ref{thmmain}} \label{S:ProofMain}
In this section we deduce Theorem \ref{thmmain} from Theorem \ref{ThDetMC} and Proposition \ref{P:secondEvalue}.
\begin{proof}
Let $\brA=A-\mu_\ep(A).$
By Proposition \ref{P:secondEvalue} for each $\delta$ we can find $S$  such that
$$ \left|\sum_{n\geq S/\ep} \mu_\ep(\brA \brA\circ T_\ep^n)\right|<\delta $$
so it suffices to get the asymptotics of $\mu_\ep(\brA \brA\circ T_\ep^n)$ for $n\approx t/\ep$ where
$t\leq S.$ We need to estimate $D_n=\mu_\ep(\brA \brA\circ T_\ep^n).$ Take a natural number $n_0.$
We have
$$ D_n=\int \brA(x) \phi_\ep(x) \brA(T_\ep^n x) dx 
        =\int \brA(T_\ep^{n_0} y) \left[\L_\ep^{n-n_0} (\L_\ep^{n_0} (\brA\phi_\ep)\right](y) dy . $$
Since $L_\ep$ depends continuously on $\ep$ as a map $\bv\to L^1$ we can rewrite the last expression~as
$$ \int \brA(T_\ep^{n_0} y) \left[\L_\ep^{n-n_0} (P (\brA\phi_\ep))\right](y) dy +O(\theta^{n_0})+o_{\ep\to 0}(1). $$
Since $T_0$ is mixing we have
$$ \int \brA(T^{n_0}_\ep y) B(y) dy=\sum_k \left( \int_{I_k} \brA(y) \phi_k dy \int_{I_k} B(y) dy\right)+
\left[O(\theta^{n_0})
+o_{\ep\to 0}(1)\right]||A||_\bv ||B||_\bv. $$
As $\ep\to 0$ the first factor converges to $\bA(k)-\sum_l p_l \bA(l)$
while the second term equals
\begin{align*}
 \int_{I_k} \L_\ep^{n-2n_0} P(\brA \phi_ep)&=
\sum_j 1_{I_k} p_j \left(\bA(j)-\sum_l p_l \bA(l)\right)  \L_\ep^{n-2n_0} \phi_j dy+o_{\ep\to 0}(1)\\&=
\sum_j p_j \left(\bA(j)-\sum_l p_l \bA(l)\right) \mu_j(T^{n-2n_0}_\ep x\in I_k) .
\end{align*}
Since $n_0$ was arbitrary we conclude that
$$ \lim_{\ep\to 0} \ep \sum_{n\leq S/\ep} \mu_\ep(\brA \brA\circ T_\ep^n)=\sum_{jk}
\int_0^S \left[p_j \bA(j) \bA(k) p_{jk}(t) -\left(\sum_j p_j \bA(j)\right)^2\right] dt. $$
Since $S$ is arbitrary we can let $S\to\infty$ and obtain Theorem \ref{thmmain}.
\end{proof}

\section{Proof of Proposition \ref{P:secondEvalue}}\label{S:ProofSecond}
\begin{proof}
Recall that $\L_\ep$ satisfy uniform Lasota-Yorke Inequality, that is there is a constant $K$ such that
\begin{align}
\label{LYV} \var{\L_\ep^n A}& \leq K\left[||A||_{L^1}+\theta^n \var{A}\right] \\
\label{LYL} ||\L_\ep^n A||_{L^1} & \leq K ||A||_{L^1}
\end{align}
Consequently it suffices to show that if $A\in\bv_0$  then
\begin{equation}
\label{L1BV}
||\L_\ep^{\kappa/2\ep} A ||_{L^1}\leq (4K)^{-1} ||A||_\bv
\end{equation}
since combining \eqref{LYV}, \eqref{LYL} and \eqref{L1BV} gives
$$ ||\L_\ep^{\kappa/\ep} A ||_\bv \leq 2^{-1} ||A||_\bv. $$
Next mixing of $T_0$ on $I_k$ implies that if $\int_{I_k} A(x) dx=0$ for each $k$ then
$$ || \L_0^{n_0} A ||_{L^1}\leq C \theta^{n_0} ||A||_\bv .$$
Consequently given $\delta$ we can find $n_0, \ep_0$ and $\delta_1$ such that if
$\ep\leq \ep_0$ and $|\int_{I_k} A(x)dx|\leq \delta_1$ for each $k$ then
$$ ||\L_\ep^{n_0}||_{L^1}\leq \delta. $$
Let $\brn=\frac{\kappa}{2\ep}-n_0.$ By the foregoing discussion it remains to show that for each $k$
$$ \left| \int_{I_k} (\L_\ep^\brn A)(x) dx\right|\leq \delta_1 $$
provided that $\kappa$ is large enough. Since
$$\L_\ep^\brn A=\L_\ep^{\brn-n_1}\L_\ep^{n_1} A=\L_\ep^{\brn-n_1} (PA)+O(\theta^{n_1})+o_{\ep\to 0}(1)$$
we need to show that $\int_{I_k} (\L_\ep^{\brn-n_1} P A)(x) dx$ is small. By Theorem \ref{ThDetMC} this integral
is asymptotic to
$ \sum_j \int_{I_j} A(y) dy p_{jk}(\kappa) .$ As $\kappa\to\infty$ this expression converges to
$p_k\int_I A(y) dy=0$ so the result follows.
\end{proof}

\appendix
\section{Proof of Proposition \ref{P:escaperate}}
\label{S:appendix}

\begin{proof}
Part (a) is proven using arguments similar to the ones in \cite[Section 3.2]{KellerLiverani09}.
Namely we apply \cite[Theorem 2.1]{KellerLiverani09}
with $\cP_0=\L_0,$ $\cP_\ep=\L_{j,\ep},$ and the Banach space
$ \bv(I_j).$ This theorem says that
$$\lim_{\ep\to 0} \frac{1-\lam_{j,\ep}}{\Delta_\ep}=1-\sum_{k \ge 0} q_k$$
with
$$\Delta_\ep=\mu_j(H_{j,\ep})=\be_j \ep+o(\ep), $$
$$q_k=\lim_{\ep\to 0} \frac{\mu_j(T_0^{-1} T_\ep^{-k} H_{j, \ep})-
\mu_j(T_\ep^{-(k+1)} H_{j, \ep})}{\Delta_\ep}. $$
Now for fixed $k$ and all small $\ep$ the set $T^{-k}_\ep H_{j, \ep}$ consists of a finite number of intervals of
size $O(\ep)$ so the sizes of their preimages by $T_0$ and $T_\ep$ differ by $O(\ep^2)$
(here we are using assumption \textbf{(II)}, which implies that $\phi_0$ is continuous at points in
$T_0^{-(k+1)}H_\ep$).

This completes the proof of part (a). Parts (b) and (c) follow from  \cite{KellerLiverani99} except for
\eqref{DensityHoles} which is proven in \cite{GonzalezTokmanHuntWright}.
\end{proof}


\end{document}